\documentclass{article}
\usepackage[cp1251]{inputenc}
\usepackage[english]{babel}
\usepackage{amsmath}
\usepackage{amsthm}
\usepackage{mathtext}
\usepackage{amssymb}
\usepackage{amscd}
\usepackage{graphics}
\usepackage{euscript}
\usepackage{bbm}

\theoremstyle{plain}
\newtheorem{theorem}{Theorem}[section]

\newtheorem{prop}[theorem]{Proposition}

\theoremstyle{definition}
\newtheorem{defin}[theorem]{Definition}
\newtheorem{constr}[theorem]{Construction}

\theoremstyle{remark}

\renewcommand{\phi}{\varphi}

\newcommand{\pg}{P_{\Gamma,\varepsilon}}

\newcommand {\ib}[1]{\textit{\textbf{#1}}}

\begin{document}
\title{\bf\Large{Graph-truncations of $3$-polytopes.}}
\author{Nickolai Erokhovets \thanks{The work was supported by RFBR grant No
14-01-31398-a}}
\date{}
\maketitle

\begin{abstract} In this paper we study the operation of cutting off edges of a simple
$3$-polytope $P$ along the graph $\Gamma$. We give the criterion when the
resulting polytope is simple and when it is flag. As a corollary we prove the
analog of Eberhard's theorem about the realization of polygon vectors of
simple $3$-polytopes for flag polytopes.
\end{abstract}

\section{Introduction.}
For the introduction to the polytope theory we recommend the books
\cite{Gb03,Z07}.

\begin{defin}
A \emph{convex polytope} $P$ is a set
$$
P=\{\ib{x}\in \mathbb R^n\colon \ib{a}_i\ib{x}+b_i\geqslant 0, i=1,\dots,m\}
$$
Let this representation be \emph{irredundant}, that is deletion of any
inequality changes the set. Then each hyperplane
$\mathcal{H}_i=\{\ib{x}\in\mathbb R^n\colon \ib{a}_i\ib{x}+b_i=0\}$ defines a
\emph{facet} $F_i=P\cap \mathcal{H}_i$.

In the following by a \emph{polytope} we mean a convex polytope.

A \emph{dimension} $\dim(P)$ of the polytope $P$ is defined as $\dim {\rm
aff}(P)$. We will consider $n$-dimensional polytopes ($n$-polytopes) in
$\mathbb R^n$.

A \emph{face} $F$ of a polytope is an intersection $F=P\cap
\{\ib{a}\ib{x}+b=0\}$ for some \emph{supporting hyperplane}
$\{\ib{a}\ib{x}+b=0\}$, i.e. $\ib{a}\ib{x}+b\geqslant 0$ for all $\ib{x}\in
P$. Each face is a convex polytope itself. $0$-dimensional faces are called
\emph{vertices}, $1$-dimensional faces -- \emph{edges}, $(n-1)$-faces --
\emph{facets}. It can be shown that the set of all facets is $\{F_1,\dots,
F_m\}$. Intersection of any set of faces of polytope is a face again (perhaps
empty).

A vertex of an $n$-polytope $P$ is called \emph{simple} if it is contained in
exactly $n$ facets. An $n$-polytope $P$ is called \emph{simple}, if all it's
vertices are simple. Each $k$-face of a simple polytope is an intersection of
exactly $n-k$ facets.

A \emph{combinatorial polytope} is an equivalence class of combinatorially
equivalent convex polytopes, where two polytopes are \emph{combinatorially
equivalent} if there is an inclusion-preserving bijection of the sets of
their faces.

A simple polytope is called \emph{flag} if any set of pairwise intersecting
facets $F_{i_1},\dots,F_{i_k}$: $F_{i_s}\cap F_{i_t}\ne\varnothing$ has
nonempty intersection $F_{i_1}\cap\dots\cap F_{i_k}\ne\varnothing$.

A \emph{non-face} is the set $\{F_{i_1},\dots, F_{i_k}\}$ with
$F_{i_1}\cap\dots\cap F_{i_k}=\varnothing$. A \emph{missing face} is an
inclusion-minimal non-face.
\end{defin}
The following results are well-known.
\begin{prop} A polytope $P$ is flag if and only if all its missing faces have
cardinality $2$.
\end{prop}
\begin{prop}
Each face of a flag polytope is a flag polytope again.
\end{prop}
\begin{prop}\label{3-belt}
The simplex $\Delta^n$ is not flag for $n\geqslant 3$. A $3$-polytope
$P^3\ne\Delta^3$ is not flag if and only if it has missing face of
cardinality $3$: $\{F_i,F_j,F_k\}$, $F_i\cap F_j, F_j\cap F_k, F_k\cap F_i
\ne\varnothing$, $F_i\cap F_j\cap F_k=\varnothing$.
\end{prop}
\begin{defin}
Missing face $\{F_i,F_j,F_k\}$ of a $3$-polytope $P^3$ we will also call a
\emph{$3$-belt}.
\end{defin}
Let $f_i(P)$ be the number of $i$-faces of the polytope $P$.
\begin{prop}[The Euler formula]
For a $3$-polytope we have
$$
f_0-f_1+f_2=2
$$
\end{prop}
Let $p_k$ be the number of $2$-faces of $P$ that are $k$-gons.
\begin{prop}
For a simple $3$-polytope $P^3$ we have
\begin{equation*}
3p_3+2p_4+p_5=12+\sum\limits_{k\geqslant 7}(k-6)p_k\qquad\qquad(*)
\end{equation*}
\end{prop}
\begin{proof}
Let us count the number of pairs (edge, it's vertex). It is equal to $2f_1$,
and since $P$ is simple to $3f_0$. Then $f_0=\frac{2f_1}{3}$ and from the
Euler formula we obtain $2f_1=6f_2-12$. Then counting the pairs (facet, it's
edge) we have
$$
\sum\limits_{k\geqslant 3}kp_k=2f_1=6\left(\sum\limits_{k\geqslant 3}p_k\right)-12,
$$
which implies the formula (*).
\end{proof}
\begin{theorem}[Eberhard]{\rm  \cite{Eb1891}}\label{ET}{ }\\
For every sequence $(p_k|3\leqslant k\ne 6)$ of nonnegative integers
satisfying {\rm (*)}, there exist values of $p_6$ such that there is a simple
$3$-polytope $P^3$ with  $p_k=p_k(P^3)$ for all $k\geqslant 3$.
\end{theorem}
\section{Graph-truncations}
\begin{constr}
Consider a subgraph $\Gamma$ without isolated vertices in the edge-vertex
graph $G(P)$ of a simple $3$-polytope $P$. For each edge $E_{i,j}=F_i\cap
F_j=P\cap \{\ib{x}\in\mathbb R^3\colon
(\ib{a}_i+\ib{a}_j)\ib{x}+(b_i+b_j)=0\}$ consider the halfspace
$\mathcal{H}_{ij,\varepsilon}^+=\{\ib{x}\in\mathbb R^3\colon
(\ib{a}_i+\ib{a}_j)\ib{x}+(b_i+b_j)\geqslant \varepsilon\}$. Set
$$
P_{\Gamma,\varepsilon}=P\cap\bigcap\limits_{E_{i,j}\in \Gamma} \mathcal{H}^+_{ij,\varepsilon}
$$
For small values of $\varepsilon$  the combinatorial type of $\pg$ does not
depend on $\varepsilon$. We will denote it $P_{\Gamma}$ and call a
\emph{graph-truncation} of $P$.

Facets of the polytope $P_{\Gamma}$ are in one-to one correspondence with
facets $F_i$ of $P$ (denote such facets by the same symbol $F_i$) and edges
$F_i\cap F_j\in\Gamma$ (denote such facets as $F_{i,j}$).
\end{constr}
\begin{prop}\label{Simple}
The polytope $P_{\Gamma}$ is simple if and only if the graph $\Gamma$ does
not contain vertices of valency $2$.
\end{prop}
\begin{proof}
For small $\varepsilon$ all new vertices of the polytope $\pg$ lie in small
neighborhoods of vertices of $P$. Consider a vertex $\ib{v}=F_{i_1}\cap
F_{i_2}\cap F_{i_3}$ of $P$ and introduce new coordinates in $\mathbb R^3$ by
the formulas
$$
y_1=\ib{a}_{i_1}\ib{x}+b_{i_1},\quad y_2=\ib{a}_{i_2}\ib{x}+b_{i_2},\quad y_3=\ib{a}_{i_3}\ib{x}+b_{i_3}.
$$
In new coordinates in some neighborhood $U(\ib{v})$ of $\ib{v}=\mathbf{0}$
the polytope $\pg$ has irredundant representation
$$
\pg\cap U(\ib{v})=\{\ib{y}\in\mathbb R^3\colon y_1,y_2,y_3\geqslant 0;y_p+y_q\geqslant
\varepsilon,\text{ if }F_{i_p}\cap F_{i_q}\in\Gamma\}
$$

The polytope $\pg$ has non-simple vertex in  $U(\ib{v})$ if and only if there
is some point in $\pg\cap U(\ib{v})$ that belongs to $4$ facets.

If $\ib{v}\notin\Gamma$ then $\ib{v}\in \pg$ is the only vertex in
$U(\ib{v})$ and it is simple.

If $\ib{v}$ has valency $1$ in $\Gamma$, say $F_{i_1}\cap F_{i_2}\in \Gamma$,
and some point lies in  $F_{i_1},F_{i_2},F_{i_3},F_{i_1,i_2}$, then
$y_1=y_2=y_2=0$ and $y_1+y_2=\varepsilon$. Contradiction.

If $\ib{v}$ has valency $2$ in $\Gamma$, say $F_{i_1}\cap F_{i_2},F_{i_2}\cap
F_{i_3}\in \Gamma$, then the point $(0,\varepsilon,0)\in\pg$ belongs to
$F_{i_1},F_{i_3},F_{i_1,i_2},F_{i_2,i_3}$, so $\pg$ is not simple.

Let $\ib{v}$ has valency $3$ in $\Gamma$ and some point belongs to $4$
facets. If there are $F_{i_1,i_2}$, $F_{i_2,i_3}$ and $F_{i_3,i_1}$, among
them, then $y_1=y_2=y_3=\frac{\varepsilon}{2}$, and there can not be neither
$F_{i_1}$, nor $F_{i_2}$, nor $F_{i_3}$. Therefore there should be at least
two of $F_{i_1},F_{i_2}, F_{i_3}$, say $F_{i_1}, F_{i_2}$. Then $y_1=y_2=0$.
But $y_1+y_2\geqslant \varepsilon$. Contradiction. So $\pg$ has only simple
vertices in $U(\ib{v})$ in this case.
\end{proof}
\begin{theorem}\label{Gflag}
A simple $3$-polytope $P_{\Gamma}$ is flag if and only if any triangular
facet of $P$ contains no more than one edge in $\Gamma$ and for any $3$-belt
$(F_i,F_j,F_k)$ of $P$ one of the edges $F_i\cap F_j$, $F_j\cap F_k$,
$F_k\cap F_i$ belongs to $\Gamma$.
\end{theorem}
\begin{proof}
Since $P_{\Gamma}$ is simple, Proposition \ref{Simple} implies that valency
of  each vertex of $\Gamma$ is $1$ or $3$.

If $P$ contains a $3$-belt $(F_i,F_j,F_k)$, such that $F_i\cap F_j$, $F_j\cap
F_k$, $F_k\cap F_i\notin \Gamma$, then $(F_i,F_j,F_k)$ is either a $3$-belt
in $\pg$. Consider a triangular face of $P$. If exactly two it's edges belong
to $\Gamma$, then Proposition \ref{Simple} implies that valency of their
common vertex is $3$ and other vertices have valency $1$ in $\Gamma$. If all
tree edges belong to $\Gamma$, then all their vertices have valency $3$ in
$\Gamma$. In both cases after truncation the face remains to be triangular,
so $P_{\Gamma,\varepsilon}$ in not flag. Thus we proved the only if part of
the theorem .

$\pg\ne\Delta^3$, since it contains more than $4$ facets. Therefore if it is
not flag, then there is a $3$-belt $(G_1,G_2,G_3)$ in $\pg$ by Proposition
\ref{3-belt}.

If $G_1=F_i$, $G_2=F_j$, $G_3=F_k$, then either $(F_i,F_j,F_k)$ is a $3$-belt
in $P$, or there is a vertex $\ib{v}=F_i\cap F_j\cap F_k\in P$. In the first
case one of the edges $F_i\cap F_j$, $F_j\cap F_k$, or $F_k\cap F_i$ belongs
to $\Gamma$ and is cut off when we pass to $\pg$, so the corresponding facets
do not intersect in $\pg$. In the second case the vertex $\ib{v}$ is cut off,
since $F_i\cap F_j\cap F_k=\varnothing$ in $\pg$. It is possible only if we
cut off one of the edges containing this vertex, so the corresponding two
facets do not intersect in $\pg$.

If $G_1=F_i$, $G_2=F_j$, and $G_3=F_{p,q}$ correspond to an edge
$E_{p,q}=F_p\cap F_q$ of $P$, then $F_i\cap F_j\ne\varnothing$ and $E_{p,q}$
intersects both $F_i$ and $F_j$. Since $F_i\cap F_j$ was not cut off, we have
$\{i,j\}\ne\{p,q\}$. If $i\in\{p,q\}$ or $j\in\{p,q\}$, then the edge
$F_p\cap F_q$ intersects the edge $F_i\cap F_j$ at the vertex, so the facets
$F_i$, $F_j$, and $F_{p,q}$ have common vertex in $\pg$. Now let $\{i,j\}\cap
\{p,q\}=\varnothing$. Then $(F_i,F_j,F_p)$ or $(F_i,F_j,F_q)$ is a  $3$-belt
in $P$. Otherwise $F_i\cap F_j\cap F_p\ne\varnothing$, $F_i\cap F_j\cap
F_q\ne\varnothing$, $F_p\cap F_q\cap F_i \ne\varnothing$, and $F_p\cap
F_q\cap F_j\ne\varnothing$, therefore $P=\Delta^3$, all it's facets are
triangles and in any triangle no more than one edge is cut off. Since
$F_i\cap F_j\notin\Gamma$ and $F_p\cap F_q\in\Gamma$, in facets $F_p$ and
$F_q$ the only edge $F_p\cap F_q$ is cut off and $\Gamma$ contains no other
edges. Then the facets $F_p$ and $F_q$ are triangles in $\pg$ either, and it
is not flag. By assumption one of the edges of the $3$-belt we obtain belongs
to $\Gamma$. Since the edge  $F_i\cap F_j$ was not cut off, one of the edges
$F_i\cap F_p$, $F_j\cap F_p$, $F_i\cap F_q$ and $F_j\cap F_q$ belongs to
$\Gamma$ and was cut off, say $F_i\cap F_p$. Then $F_i\cap
F_{p,q}=\varnothing$, which is a contadiction.

If only one of the facets $(G_1,G_2,G_3)$ corresponds to a facet of $P$, say
$G_1=F_i$, then two other facets correspond to edges of $P$ that both
intersect $F_i$ and have common vertex. If both edges belong to $F_i$, then
$F_i\cap G_2\cap G_3\ne\varnothing$. If exactly one of them belong to $F_i$,
say corresponding to $G_2$, then $F_i\cap G_3=\varnothing$. At last, if both
of them do not belong to $F_i$, then their common vertex $\ib{v}$ do not
belong to $F_i$, and these two edges and their common vertex define some
facet $F_j$ that has with $F_i$ two common vertices -- the remaining ends of
two edges, thus $F_i\cap F_j$ is an edge, connecting these vertices. Then
$F_j$ is a triangle containing two edges in $\Gamma$. Contradiction.

At last if all three facets of $3$-belt correspond to edges of $P$, then
these edges pairwise intersect. Two of them define some facet $F_i$. If the
third edge does not belong to $F_i$, then all three edges have common vertex
and in $\pg$ the corresponding facets have a common vertex either. If the
third edge belongs to $F_i$, then $F_i$ is a triangle with three edges in
$\Gamma$. Contradiction.

Thus we have considered all possible cases, and the theorem is proved.
\end{proof}
\section{Application}
As an application of Theorem \ref{Gflag} we prove an analog of Eberhard's
theorem for flag $3$-polytopes. Since any face of a flag polytope is flag
itself, we have $p_3(P^3)=0$ for any flag polytope.
\begin{theorem}
For every sequence $(p_k|3\leqslant k\ne 6)$ of nonnegative integers
satisfying $p_3=0$ and {\rm (*)}, there exist values of $p_6$ such that there
is a flag simple $3$-polytope $P^3$ with  $p_k=p_k(P^3)$ for all $k\geqslant
3$.
\end{theorem}
\begin{proof}
From Eberhard's Theorem \ref{ET} it follows that there exist values of $p_6$
such that there is a polytope $P^3$ with $p_k=p_k(P)$ for all $k\geqslant 3$.
Let us consider the graph $\Gamma=G(P^3)$. Since $p_3=0$, we obtain from
Theorem \ref{Gflag} that $P_{G(P)}$ is flag. On the other hand, facets of
$P_{G(P)}$ are in one-to-one correspondence with facets and edges of $P$.
Moreover, $k$-gonal facets of $P$ correspond to $k$-gonal facets of
$P_{G(P)}$ and edges of $P$ correspond to $6$-gonal facets of $P_{G(P)}$.
Therefore
$$
p_k(P_{G(P)})=\begin{cases}p_k(P),&k\ne 6;\\
p_6(P)+f_1(P),&k=6.
\end{cases}
$$
This proves the theorem.
\end{proof}
The author is grateful to professor V.M.~Buchstaber for encouraging
discussions and a permanent attention to his work.


\begin{thebibliography}{99}
\bibitem[Eb1891]{Eb1891} V.~Eberhard,\,\emph{Zur Morphologie der
    Polyheder}, Leipzig 1891.
\bibitem[Gb03]{Gb03} Branko Gr\"{u}nbaum,\,\emph{Convex
    polytopes}, Graduate texts in Mathematics 221, Springer-Verlag, New York,
    2003.
\bibitem[Z07]{Z07}G\"{u}nter M. Ziegler,\,\emph{Lectures on polytopes},
    Graduate texts in Mathematics 152, Springer, 2007.
\end{thebibliography}
\end{document}